\documentclass{amsart}
\usepackage{amsfonts}
\usepackage{graphicx}
\usepackage{amscd}
\usepackage{amsmath}
\usepackage{amssymb}

\makeatletter
\@namedef{subjclassname@2010}{%
	\textup{2010} Mathematics Subject Classification}
\makeatother

\theoremstyle{plain}

\newtheorem{corollary}{\bf Corollary}
\newtheorem{definition}{\bf Definition}
\newtheorem{lemma}{\bf Lemma}

\newtheorem{remark}{Remark}

\newtheorem{theorem}{\bf Theorem}

\numberwithin{equation}{section}

\usepackage{amssymb}

\usepackage{enumitem}

\usepackage{graphicx}

\usepackage{color}
\usepackage{color}
\usepackage{color}
\usepackage{color}
\usepackage{color}

\newcommand{\R}{\mathbb{R}}

\makeatletter
\@namedef{subjclassname@2010}{%
  \textup{2010} Mathematics Subject Classification}
\makeatother

\usepackage[T1]{fontenc}
\theoremstyle{definition}
\numberwithin{equation}{section}

\begin{document}

\baselineskip=17pt

%%%%%%%%%%%%%%%%

\title[Conformal gradient vector fields on  Riemannian manifolds]{Conformal Vector Fields and the De-Rham Laplacian on a Riemannian Manifold with Boundary}

\author[A. Freitas, I. Evangelista]{Antônio Freitas$^1$, Israel Evangelista$^2$}
\address{$^1$Universidade Federal do Ceará- UFC, Fortaleza /CE, Brazil}
\email{antoniofreitas1964@live.com}

%\author[I. Evangelista]{Israel Evangelista}
\address{$^2$Universidade Federal do Delta do Parna\'iba- UFDPar, Curso de Matem\'atica, Parna\'iba /PI, Brazil}
\email{israelevangelista@ufpi.edu.br}

\author[E. Viana]{Emanuel Viana$^3$}
\address{$^3$Instituto Federal de Educa\c c\~ao, Ci\^encia e Tecnologia do Cear\'a (IFCE), Campus Caucaia, Caucaia /CE, Brazil.}
\email{emanuel.mendonca@ifce.edu.br}
\date{December 21, 2021}
%\date{}

\begin{abstract}
Let $(M^n,\,g)$ be an $n$-dimensional compact connected Riemannian manifold with smooth boundary. In this article, we study the effects of the presence of a nontrivial conformal vector field on $(M^n,\,g)$. We used the well-known de-Rham Laplace operator and a nontrivial solution of the famous Fischer–Marsden differential equation to provide two characterizations of the hemisphere  $\mathbb{S}^n_+(c)$ of constant curvature $c>0$. As a consequence of the charac-terization using the Fischer–Marsden equation, we prove the cosmic no-hair conjecture under a given integral condition.
\end{abstract}



\maketitle

\section{Introduction}
Let $(M^n,\,g),$ $n\geq 2,$ be an $n$-dimensional compact smooth oriented Riemannian manifold with smooth boundary $\partial M$ and we denote by $\nabla,$ $\nabla^2,$ $\Delta$ and $dM$ the Riemannian connection, the Hessian, the Laplacian and the volume form on $M$, res-pectively, while considering them over $\partial M$ we will explicitly mention the boundary. We also denote by $h(X,Y)=g(\nabla_X\,\nu,Y),\,\,X,Y\in \mathfrak{X}(M),$ the second fundamental form associated to the unit outward normal vector field $\nu$ along $\partial M$, where  $\mathfrak{X}(M)$ is the Lie algebra  of smooth vector field on $M.$

In the middle of the last century many geometers tried to prove a conjecture concerning the Euclidean sphere as the unique compact orientable Riemannian manifold $(M^n,\,g)$ admitting a metric of constant scalar curvature $R$ and carrying a nontrivial conformal vector field $\xi$. Among them, we cite Bochner, Goldberg \cite{Goldberg1, Goldberg2}, Hsiung \cite{Hsiung}, Lichnerowicz, Nagano \cite{Nagano}, Obata \cite{Obata1} and Yano \cite{ny}; we address the reader to the book of Yano \cite{ Yano-Book} to a summary of those results. Despite of many effort to prove the conjecture, it remained opened until 1981 when Ejiri \cite{Ejiri} found  a counterexample to this conjecture building metrics of constant scalar curvature on warped product of type $\mathbb{S}^{1}\times_{h}N$, where $N$ is an $(n-1)$-dimensional Riemannian manifold with positive constant scalar curvature, $h$ is a positive function on a circle $\mathbb{S}^{1}$ satisfying a certain ordinary differential equation and $\xi=h\frac{\partial}{\partial t}$ is a  conformal vector field, see \cite{Ejiri} for details.

We recall that a smooth vector field $\xi\in  \mathfrak{X}(M)$ is said to be \textit{conformal} if
\begin{equation}\label{eqconformalfield}
	\mathcal{L}_{\xi}g=2f g
\end{equation}
for a smooth function $f$ on $M$, where $\mathcal{L}_{\xi}$ is the Lie derivative in the direction of $\xi$. The function $f$ is the conformal factor of $\xi$ (cf.~\cite{besse}). If $\xi$ is the gradient of a smooth function  on $M$, then $\xi$ is said to be a \textit{conformal gradient vector field}. In this case,  $\xi$ is  also closed. We say that $\xi$ is a nontrivial conformal vector field if it is a non-Killing conformal vector field.

An interesting  problem in Riemannian geometry is to find characterizations of spheres and hemispheres in the class of compact connected Riemannian manifolds with empty and non-empty boundary, respectively (see, e.g.,~\cite{abb, Deshmukh,deshmuk2,deshmuk, deshmuk3, evviana, hall,hall47}).

For example, one of such characterizations of spheres was given by  Obata~\cite{obata}, namely, a necessary and sufficient condition for an $n$-dimensional complete Riemannian manifold $(M^n,\,g)$ to be isometric to the $n$-sphere $\mathbb{S}^n(c)$ of constant curvature $c$ is that there exists a nonconstant smooth function $f$ on $M$ that satisfies $\nabla_{X} f=-cfX,\,X\in \mathfrak{X}(M),$ for some constant $c>0$, where $\nabla_{X}$ is the covariant derivative operator with respect to $X\in \mathfrak{X}(M)$. 

On the other hand, Reilly~\cite{reilly} proved that a  compact Riemannian manifold $M$ with totally geodesic boundary, which admits a nonconstant function $f$ on $M$ such that  $$\nabla^2 f=-c fg,$$ for some constant $c>0$, $f\geq 0$ on $M$ and $f=0$ on $\partial M$, is necessarily isometric to a hemisphere of $\mathbb{S}^{n}(c)$. In the same direction, Reilly~\cite{Reilly2} also proved that if a  compact, connected, oriented Riemannian manifold $M$ with connected nonempty boundary $\partial M$ that admits a nonconstant function $f$ on $M$ which  satisfies  $\nabla^2 f=-c fg$, for some constant $c>0$, and $f|_{\partial M}$ is constant, then $M$  is isometric to a geodesic ball on $\mathbb{S}^n(c)$.

Moreover, on a Riemannian manifold $(M^n,\,g)$, the Ricci operator $S$ is defined using Ricci tensor $Ric$ (see~\cite{besse}) by
\begin{equation}
Ric(X,Y)=g(SX,Y), \,\,\,X,Y\in\mathfrak{X}(M),
\end{equation}
and similarly, the rough Laplace operator on $(M^n,\,g)$, 
$\Delta:\mathfrak{X}(M) \to \mathfrak{X}(M),$  is defined by (cf.~\cite{duggal})
\begin{equation}
\Delta X=\sum_{i=1}^{n}(\nabla_{e_i} \nabla_{e_i} X-\nabla_{{\nabla}_{e_i}}e_i X), \,\,\,X\in\mathfrak{X}(M),
\end{equation}
where $\{e_1, \ldots, e_n\}$ is a local orthonormal frame on $M^n$. The rough Laplace operator is used to provide characterizations of the sphere as well as the Euclidean space.  (cf.~\cite{duggal,Garcia}). On Riemannian manifold $(M^n,\,g),$ we recall that de-Rham Laplace operator $\square: \mathfrak{X}(M) \to \mathfrak{X}(M)$  which is defined by (cf.~\cite{duggal}, p.83)
\begin{equation}
\square = S+\Delta
\end{equation}
is used to characterize a Killing vector field on a compact Riemannian manifold. It is important to mention that the equation $\square \xi=0$  is satisfied  if $\xi$ is a Killing vector field on a Riemannian manifold $(M, g)$ or a soliton vector field of a Ricci soliton $(M,\,g,\,\xi,\,\lambda)$ (cf.~\cite{deshmuk2}). In addition, Deshmukh, Ishan and Vilcu~\cite{deshmuk1} found two new characterizations of the $n$-dimensional sphere $\mathbb{S}^n(c)$ of constant curvature $c$, the first result being using the well-known de-Rham Laplace operator, whose presentation was made previously, while the second uses a nontrivial solution of the famous Fischer–Marsden differential equation (see~\cite{Fischer}), and which will be presented shortly after the statement of Theorem \ref{theorem1} (see Equation \eqref{FM}).

Next consider the hemisphere $\mathbb{S}^n_+(c)$ of constant curvature $c$ as a hypersurface  of the Euclidean space $\R^{n+1}$ provided with a unit normal $\xi $ and shape operator $B=-\sqrt{c}I$, where $I$ is the identity operator. One denotes by $S$ the Ricci operator of $\mathbb{S}^n_+(c)$, it is renowned that it can be written as
$$S=(n-1)cI.$$
On the other hand, we consider the unit vector field $\eta= e_{n+1}$, where  $e_{n+1}$ stands for the $(n+1)$-th vector of the canonical  basis in $\R^{n+1}$. Now, restricting $\eta$ to the hemisphere $\mathbb{S}^n_+(c)$, one  can  decompose  $\eta$ as follows 
\begin{equation}\label{eqdecompeta}
\eta=\textbf{u}+f\xi,
\end{equation}
where $f=\langle\eta,\xi\rangle$ and $\textbf{u}$ is the tangential projection of $\eta$ on the hemisphere $\mathbb{S}^n_+$ and $\langle \cdot,\cdot\rangle$ stands for the canonical metric of the  Euclidean space $\R^{n+1}$.
Given a smooth vector field $X$ on $\mathbb{S}^n_+$ and denoting by $\nabla $ the Riemannian connection on $\mathbb{S}^n_+$ with respect to the canonical metric, one can use \eqref{eqdecompeta} to calculate $\nabla_X \textbf{u}$ and the gradient of the smooth function $f$ on $\mathbb{S}^n_+(c)$  denoted by $\nabla f$. In fact, it is enough to take covariant derivative in \eqref{eqdecompeta} and use the Gauss-Weingarten equations to get
\begin{equation}\label{eqdnablau}
\nabla_X \textbf{u}=-\sqrt{c}fX, \,\,\, \nabla f=\sqrt{c}\textbf{u}.
\end{equation}
Now we can use \eqref{eqdnablau} to obtain   an expression  for the rough Laplacian acting on $\textbf{u}$ and the Laplace operator acting on $f$ as follows
\begin{equation}\label{eqlaplaceuf}
\triangle \textbf{u}=-c\textbf{u}, \,\,\, \triangle f=-ncf.
\end{equation}
Hence, one can use \eqref{eqlaplaceuf} to conclude that
\begin{equation}\label{eqquadradou}
\square \textbf{u}=(n-2)c\textbf{u}.
\end{equation}
Furthermore, due to our choice on $\eta$ we have $f|_{\partial \mathbb{S}^n_+(x)}=0$ and in addition  with equation \eqref{eqdnablau} we conclude that $f$ is not a constant  function and $\textbf{u}$ is not a parallel vector field. Now we may state our first main result. More precisely, 

\begin{theorem}\label{theorem1}
	Let $(M^n,\,g),$ $n > 2$, be a smooth compact Riemannian manifold with smooth totally geodesic boundary $\partial M$ and $\xi$ a smooth conformal vector field on $M$ with conformal factor $f$ satisfying $f|_{\partial M}=0$. Then $\square \xi=\lambda \xi$ for a constant $\lambda,$ if and only if, $\lambda>0$ and $(M^n,\,g)$ is isometric to the hemisphere  $\mathbb{S}^{n}_+\left(\dfrac{\lambda}{n-2}\right).$ 
\end{theorem}

In addition, Fischer–Marsden considered the following differential equation (see \cite{Fischer}) on a Riemannian manifold $(M^n,\,g)$:
\begin{equation}\label{FM}
	(\Delta f)g+f Ric=\nabla^2 f.
\end{equation}
It is known that if a complete Riemannian manifold $(M^n,\,g)$ has a nontrivial solution $f$ to (\ref{FM}), then the scalar curvature $R$ of $g$ is a constant (see \cite{Bourguignon,Fischer}). We remark that Fischer and Marsden conjectured that if a compact Riemannian manifold admits a nontrivial solution of the differential Equation (\ref{FM}), then it must be an Einstein manifold.
Counterexamples to the conjecture were provided by Kobayashi \cite{Kobayashi}  and Lafontaine \cite{Lafontaine}.

Furthermore, recall that if an $n$-dimensional Riemannian manifold $(M,\,g)$ admits a nontrivial solution of the Fischer–Marsden differential equation \eqref{FM}, $n > 2$, then the non-trivial function $f$ satisfies
\begin{equation}
	\Delta f=-\frac{R}{n-1}f.
\end{equation}

Now, we consider an $n$-dimensional Riemannian manifold $(M,\,g)$ that admits a nontrivial conformal vector field $\xi$ with conformal factor $f$ that is a non-trivial solution of
the Fischer–Marsden differential equation \eqref{FM} and define a constant $\alpha$ by $R = n(n - 1)\alpha$ for this Riemannian manifold. Hence, we deduce $\Delta f=-n\alpha f$.
  
Thereby, we obtain the following characterization of the hemisphere, namely, 

\begin{theorem}\label{theorem2}Let $\xi$ be a nontrivial conformal vector field with conformal factor $f$ such that $f=0$ on $\partial M$  and associated operator $\varphi$ on an $n$-dimensional Riemannian manifold $(M^n,\, g)$, $n > 2$, with smooth  totally geodesic boundary $\partial M$. If $f$ is a nontrivial solution of the Fischer–Marsden equation \eqref{FM}, then
	\begin{equation}\label{eqtheorem2}
		\int_{M}Ric(\nabla f+\alpha \xi,\nabla f+\alpha \xi)dM\leq \alpha^2\int_{M}|\varphi|^2dM,
	\end{equation}
	holds for a constant $\alpha,$ defined by $R=n(n-1)\alpha$. Moreover, equality holds if and only if $M^n$ is isometric to the hemisphere $\mathbb{S}_{+}^n(\alpha)$.
\end{theorem}

Now we define the static triple to get an interesting corollary of the previous theorem. 

\begin{definition}\label{DefStatic}
	A complete and connected Riemannian manifold $(M^n, g)$ with a (possibly nonempty) boundary $\Sigma$ is said to be static if there exists a non-negative function $f$ on $M$ satisfying
	\begin{equation}\label{static}
		(\Delta f)g+f Ric =\nabla^2 f,
	\end{equation}in $M \setminus\Sigma$ and $\Sigma = f^{-1}(0)$. In this case $(M^n,\,g,\,f)$ is called a static triple or simply a static metric.
\end{definition}

We remark that Eq. (\ref{static}) appears in General Relativity, where it defines static solutions of Einstein field equations. Corvino et al. [~\cite{CM}, Proposition 2.1] showed that a static metric also has constant scalar curvature $R$. When the scalar curvature is positive there exists a classic conjecture called \textit{cosmic no-hair conjecture}, formulated by Boucher et al.~\cite{Bou} which claims that:

\vspace{0,4cm}
\textit{The only $n$-dimensional compact static triple $(M^n,\,g,\,f)$ with positive scalar curvature and connected boundary $\Sigma$ is given by a round hemisphere $\mathbb{S}^{n}_{+}$, where the function $f$ is taken as the height function.}
\vspace{0,4cm}

As a consequence of the Theorem \ref{theorem2}, we get a partial answer to the cosmic no-hair conjecture, in other words, we get the following:

\begin{corollary}\label{coro1}
Let $(M^n,\,g,\,f)$ a $n$-dimensional compact static triple  with scalar curvature $R=n(n-1)\alpha$ and smooth totally geodesic boundary $\partial M$ and $\xi$ be a nontrivial conformal vector field with conformal factor $f$ and associated operator $\varphi$. Supposing $\displaystyle\int_{M}Ric(\nabla f+\alpha \xi,\nabla f+\alpha \xi)\geq \alpha^2\displaystyle\int_{M}|\varphi|^2,$ then $(M^n,\,g,\,f)$ is isometric to the hemisphere $\mathbb{S}^{n}_{+}(\alpha)$.
\end{corollary}

From Theorem \ref{theorem2} and the integral hypothesis of Corollary \ref{coro1}, we have $$\displaystyle\int_{M}Ric(\nabla f+\alpha \xi,\nabla f+\alpha \xi)dM= \alpha^2\displaystyle\int_{M}|\varphi|^2dM,$$ and, therefore, the result of the corollary follows.

This article is organized as follows. In Section \ref{prel}, we review some classical tensors and basic facts about conformal vector fields. Moreover, we present some key lemmas that will be used in the proof of the main results. In Section \ref{mr}, we prove Theorem \ref{theorem1} and Theorem \ref{theorem2}.

\section{Preliminaries}\label{prel}

In this section, we present basic facts that will be useful to obtain the main results. Remember from equation \eqref{eqconformalfield} that a vector field $\xi$ on Riemannian manifold $(M^n,\,g)$ is called conformal if $\mathcal{L}_{\xi}g$ is a multiple of $g.$ As a straightforward consequence of Koszul's formula, we have the following identity for any smooth vector field $Z$ on $M$,
\begin{equation*}%\label{eqkozul}
	2g(\nabla_X Z,Y)=\mathcal{L}_{Z}g(X,Y)+d\eta (X,Y), \,\,\,X,Y\in\mathfrak{X}(M),
\end{equation*}
where $\eta$ stands for the dual 1-form associated to $Z$, that is, $\eta(Y)=g( Z,Y )$. We note that we can define $\varphi$
the following skew symmetric (1,1)-tensor:
\begin{equation*}%\label{eqdeta}
	d\eta(X,Y)=2g(\varphi(X),Y),\,\,\,X,Y\in\mathfrak{X}(M).
\end{equation*}
Therefore,  one can use the above equations to get
\begin{equation}\label{eqconf2}
	\nabla_X\xi=\Phi(X), \,\,\, X\in \mathfrak{X}(M),
\end{equation} where $\Phi:=fI+\varphi(\cdot)$ and $I$ is the Identity Operator. Note that $\Phi$ gives us an idea of how much of the field $\xi$ is not closed vector field. There are several papers involving closed conformal vector fields  with many  authors working on it  (see, e.g.,~\cite{caminha, hicks, tanno}).

Note that we can identify $\varphi$ with a  skew symmetric $(0,2)$-tensor and $\xi$ with the tensor $\xi(Y)=g(\xi,Y), \,\,Y\in \mathfrak{X}(M)$, to rewrite \eqref{eqconf2} as follows
\begin{equation}\label{eqderxitens}
	\nabla \xi=fg+\varphi.
\end{equation}

Moreover, we adopt the following expression for curvature tensor
\begin{equation}
R(X,Y)Z=\nabla_X \nabla_Y Z-\nabla_Y \nabla_X Z-\nabla_{[X,Y]}Z. \nonumber
\end{equation} 
Whence, one can use  Equation \eqref{eqconf2} to get
\begin{equation}
R(X,Y)\xi=X(f)Y-Y(f)X+(\nabla \varphi)(X,Y)-(\nabla \varphi)(Y,X),
\end{equation}
where $(\nabla \varphi)(X,Y)=\nabla_X \varphi Y -\varphi(\nabla_X Y).$

Using the above equation and the expression for the Ricci tensor
$$Ric(X,Y)=\sum_{i=1}^n g(R(e_i,X)Y,e_i),$$
where $\{e_1, \ldots, e_n\}$ is a local orthonormal frame, we obtain
\begin{equation}\label{ric1}
Ric(\xi,Y)=-(n-1)Y(f)-\sum_{i=1}^n g(Y,(\nabla \varphi)(e_i,e_i)),
\end{equation}
where we used the skew symmetry of the operator $\varphi$. The above equation gives
\begin{equation}\label{ric2}
S(\xi)=-(n-1)\nabla f-\sum_{i=1}^n (\nabla \varphi)(e_i,e_i).
\end{equation}

Now, using Equation \eqref{eqconf2}, we compute the action of the rough Laplace operator $\Delta$ on
the vector field $\xi$ and find
\begin{equation}\label{rough laplace}
\Delta\xi=\nabla f+\sum_{i=1}^n(\nabla \varphi)(e_i,e_i).
\end{equation}

The following lemma, obtained previously in [~\cite{evviana}, Lemma 1], will be useful in proof of Theorem \ref{theorem1}.

\begin{lemma}\label{lemmadiv}Let $(M^n,\,g)$ be a smooth compact Riemannian manifold with smooth totally geodesic boundary $\partial M$ and $\xi$ a smooth conformal vector field on $M$ with conformal factor $f$ satisfying $f|_{\partial M}=0$. Denote by ${\rm div}$ and ${\rm div}_{\partial M}$ the divergence operators on $M$ and $\partial M$, respectively, and
	by $\xi^T$ the tangential part of $\xi$ on $\partial M$. Then,
	\begin{equation}\label{eqdiv}
	{\rm div}(\xi)=nf,\,\,\,\,\,\,\,\ {\rm div}_{\partial M}(\xi^T)=(n-1)f.
	\end{equation}
	Furthermore,
	\begin{equation}\label{eqint}
	\int_M g(\nabla f,\xi) dM=-n\int_M f^2 dM.
	\end{equation}
\end{lemma}

Moreover, the second and third authors obtained integral expressions involving the conformal field and the conformal factor (see \cite{evviana}, Lemmas 2.1 and 2.4). More precisely, they proved 

\begin{lemma}\label{lemmarici}Let $(M^n,\,g)$ be a smooth compact Riemannian manifold with  smooth boundary $\partial M$ and constant scalar curvature $R$. Let $\xi$ be  a smooth conformal vector field on $M$ with conformal factor $f$ such that $f=0$ on $\partial M$. Then,
	\begin{equation}\label{eqricci}
		\int_M Ric(\xi,\nabla f)dM=-R\int_Mf^2 dM.
	\end{equation}
\end{lemma}

\begin{lemma}\label{lemmarici2}Let $(M^n,\,g)$ be a smooth compact Riemannian manifold with smooth totally geodesic boundary $\partial M$. Let $\xi$ be  a smooth conformal vector field on $M$ with conformal factor $f$ such that $f=0$ on $\partial M$ and $\Delta f=-n\alpha f$. Then,
	\begin{equation}\label{eqricci2}
		\int_M Ric(\nabla f,\nabla f)dM=-\int_M(|\nabla^2f|^2-(\Delta f)^2)dM,
	\end{equation} 
\end{lemma}

Next, using the ideas of~\cite{deshmuk1}, we present a lemma that will be useful in proofing the results of section \ref{mr}.

\begin{lemma}\label{lemmarici3}Let $(M^n,\,g)$ be a smooth compact Riemannian manifold with smooth totally geodesic boundary $\partial M$. Let $\xi$ be  a smooth conformal vector field on $M$ with conformal factor $f$ and $\varphi$ the skew symmetric tensor associated to $\xi$ such that $f=0$ on $\partial M$. Then,
	\begin{equation}\label{eqricci3}
		\int_M Ric(\xi,\xi)dM=\int_M(|\varphi|^2+n(n-1)f^2)dM.
	\end{equation}
\end{lemma}

\begin{proof}
Using Equation \eqref{ric2}, we have
\begin{equation}
	Ric(\xi,\xi)=-(n-1)\xi(f)-\sum_{i=1}^n g(\xi,(\nabla \varphi)(e_i,e_i)),
\end{equation}and using Equation \eqref{eqconf2} and skew-symmetry of the associated operator $\varphi$, we find

\begin{equation}
	{\rm div} \varphi(\xi)=-|\varphi|^2-\sum_{i=1}^n g(\xi,(\nabla \varphi)(e_i,e_i)),
\end{equation}where $$|\varphi|^2=\sum_{i=1}^n g(\varphi\,e_i,\varphi\,e_i).$$

Hence, we get
\begin{equation}\label{eqricci3d}
Ric(\xi,\xi)=-(n-1)\xi(f)+{\rm div}(\varphi(\xi))+|\varphi|^2.
\end{equation}
Integrating (\ref{eqricci3d}) over $M,$ using the divergence theorem, the Lemma \ref{lemmadiv} and the hypothesis of $f$ over the boundary, we obtain
\begin{equation}
		\int_M Ric(\xi,\xi)dM-\int_M(|\varphi|^2+n(n-1)f^2)dM =\int_{\partial M}g(\varphi(\xi),\nu)d(\partial M).
	\end{equation}

Furthermore, since $\varphi$ is skew symmetric and $f=0$ in $\partial M,$ we have
\begin{equation}
g(\varphi(\xi),\nu)=
%-g(\xi,\varphi(\nu))=-g(\xi,\nabla_{\nu}\xi-f\nu)=-g(\xi,\nabla_{\nu}\xi)=-%5\frac{1}{2}\nu(|\xi|^2)=
-\dfrac{1}{2}g(\nu,\nabla|\xi|^2)
\end{equation} 	

Hence, since $(M^n,\,g)$ be a smooth compact Riemannian manifold with smooth totally geodesic boundary $\partial M,$ we have ${\rm div}_{\partial M}(\nu)=0,$ and likewise, $${\rm div}_{\partial M}(|\xi|^2\nu)=g(\nabla|\xi|^2,\nu),$$ and so $\displaystyle\int_{\partial M}g(\nabla|\xi|^2,\nu)d(\partial M)=0,$ which proves the Lemma.
\end{proof}

For the sake of completeness we include here a proof of following Lemma, and from an equation contained in their proof, we get that $R>0$. (cf. Remark \ref{observation})

\begin{lemma}\label{lemmarici5}Let $(M^n,\,g)$ be a smooth compact Riemannian manifold with smooth totally geodesic boundary $\partial M$. Let $\xi$ be  a smooth conformal vector field on $M$ with conformal factor $f$ such that $f=0$ on $\partial M$ and $\Delta f=-n\alpha f$. Then,
	\begin{equation}\label{eqricci5}
		\int_M Ric(\nabla f,\xi)dM=-n(n-1)\alpha\int_Mf^2 dM.
	\end{equation}
\end{lemma}

\begin{proof}
	Using the  tensorial Ricci-Bochner formula
	\begin{equation}\label{bochT}
		{\rm div}(\nabla^2 f)=Ric(\nabla f)+\nabla (\Delta f)
	\end{equation}
	and the skew-symmetry of $\varphi$, we obtain
	\begin{eqnarray*}
		{\rm div}(\nabla^2 f(\xi))
		&=&g(\nabla (\Delta f),\xi)+Ric(\nabla f,\xi)+f\Delta f.
	\end{eqnarray*}
	On the other hand, since
	${\rm div}(\Delta f \xi)=g(\nabla (\Delta f),\xi)+nf\Delta f$, we have
	\begin{equation}\label{eqdivhes1}
		{\rm div}(\nabla^2 f(\xi)-\Delta f \xi)=-(n-1)f\Delta f +Ric(\nabla f,\xi).
	\end{equation}
	Integrating \eqref{eqdivhes1} over $M$ and using the divergence theorem yields
	\begin{equation}\label{eqdivhes}
		\int_{\partial M} (\nabla^2 f(\xi,\nu)- g(\xi,\nu)\Delta f)d(\partial M)=-(n-1)\int_Mf\Delta f dM+\int_M Ric(\nabla f,\xi)dM.
	\end{equation}
	Using the Gauss-Weingarten  equations and  the fact that $\partial M$ is totally geodesic, we get
	\begin{eqnarray*}
		\int_{\partial M} (\nabla^2 f(\xi,\nu)- g(\xi,\nu)\Delta f)d(\partial M) &=& \int_{\partial M} g(\nabla _{\partial M}f_\nu,\xi^T) d(\partial M)\\
		&&-\int_{\partial M}g(\xi,\nu)\Delta _{\partial M}f  d(\partial M).
	\end{eqnarray*}
	So we can use Lemma \ref{lemmadiv} to conclude that
	\begin{eqnarray*}
		\int_{\partial M} (\nabla^2 f(\xi,\nu)- g(\xi,\nu) \Delta f)d(\partial M)&=& -\int_{\partial M} f_\nu div_{\partial M}(\xi^T) d(\partial M)\\
		&&-\int_{\partial M} f \Delta (g(\xi,\nu) )d(\partial M)\\
		&=& -(n-1)\int_{\partial M} f_\nu f d(\partial M)\\
		&&-\int_{\partial M} f \Delta( g(\xi,\nu) )d(\partial M)\\
		&=&0.
	\end{eqnarray*}	
Thus, we have
	\begin{equation}\label{eqdivhes}
		(n-1)\int_Mf\Delta f dM=\int_M Ric(\nabla f,\xi)dM.
	\end{equation}
	Since $\Delta f=-n\alpha f$,   the proof of the lemma is finished. 	
\end{proof}

\begin{remark}\label{observation}
	Since $\dfrac12\Delta f^2=|\nabla f|^2+f\overline{\Delta} f$, $\Delta f^2=2\,{\rm div}(f\nabla f)$ and  $f|_{\partial M}=0$, we easily get
	$$\int_M |\nabla f|^2dM=-\int_Mf\Delta  f\,dM.$$
	
	Thus, by Lemma \ref{lemmarici} and  \eqref{eqdivhes}, we have  $$\int_M|\nabla f|^2dM=\frac{R}{n-1}\int_Mf^2 dM.$$ Therefore, $R > 0.$
\end{remark}

\section{ Main Results }\label{mr}

In this section, we will present the proof of Theorems \ref{theorem1} and \ref{theorem2}. 

\subsection*{Proof of Theorem \ref{theorem1}:}
\begin{proof}
	Suppose $\xi$ is a nontrivial conformal vector field with conformal factor $f$ satisfying $f|_{\partial M}=0$ that satisfies $\square \xi=\lambda \xi.$ Then using equations \eqref{ric2} and \eqref{rough laplace}, we have 
	\begin{equation}
	\square \xi=S \xi+\Delta \xi=-(n-2)\nabla f.
	\end{equation}
Therefore,
\begin{equation}\label{nabla f}
\nabla f =-\frac{\lambda}{n-2}\xi.
\end{equation}	

By equation (\ref{nabla f}) and Lemma \ref{lemmadiv} we have
\begin{equation}\label{eqint}
	-\frac{\lambda}{n-2}\int_M g(\xi,\xi) dM=\int_M g(\nabla f,\xi) dM=-n\int_M f^2 dM.
	\end{equation}
If $\lambda = 0,$ then the above equation  implies that  $f=0$ in $M,$ contradicting  our assumption that $\xi$ is a nontrivial
conformal vector field. Hence, the constant $\lambda > 0$. 

Now, taking covariant derivative in Equation (\ref{nabla f}) and using Equation (\ref{eqconf2}), we get
\begin{equation}
\nabla_X \nabla f=-\frac{\lambda}{n-2}(fX+\varphi (X)), \,\,\, X\in \mathfrak{X}(M).
\end{equation}
Taking the inner product with $ X\in \mathfrak{X}(M)$ in the above equation and noticing that $\varphi$ is skew symmetric, we conclude that
	\begin{equation}
g(\nabla_X \nabla f,X)=-\frac{\lambda f}{n-2}g(X,X), \,\,\, X\in \mathfrak{X}(M).
\end{equation}
 Using polarization in above expression, and noticing that	
	\begin{equation}
\nabla^2f(X,Y)= g(\nabla_X \nabla f,Y), \,\,\, X,Y\in \mathfrak{X}(M)
\end{equation} is symmetric, we get
	\begin{equation}
\nabla^2f(X,Y)=-\frac{\lambda f}{n-2}g(X,Y), \,\,\, X,Y\in \mathfrak{X}(M).
\end{equation}
	
Therefore, by the hypothesis on $f$, we can apply  Theorem B in \cite{Reilly2} to ensure that $M$ is isometric to a geodesic ball on $\mathbb{S}^{n}\left(\frac{\lambda}{n-2}\right)$. Since  $\partial M$ is totally geodesic, we conclude that  $M$ is isometric to a  hemisphere of $\mathbb{S}^{n}\left(\frac{\lambda}{n-2}\right)$.

Conversely, if $(M,\,g)$ is isometric to the a hemisphere of  $\mathbb{S}^{n}\left(\frac{\lambda}{n-2}\right)$ then Equation \eqref{eqquadradou} confirms the existence of nontrivial vector field $\textbf{u}$ satisfying $\square \textbf{u} = \lambda \textbf{u}$ for a positive constant $\lambda$.
\end{proof}

Finally, we present the proof of our second result.

\subsection*{Proof of Theorem \ref{theorem2}:}

\begin{proof}
Since
$$Ric(\nabla f+\alpha \xi,\nabla f+\alpha \xi)=Ric(\nabla f,\nabla f)+2\alpha Ric(\nabla f,\xi)+\alpha^2Ric(\xi,\xi),$$
integrating the above expression and using the Equations \eqref{eqricci2} and \eqref{eqricci3} and the Lemma \ref{lemmarici5}, we conclude
\begin{eqnarray}
\int_M Ric(\nabla f+\alpha \xi,\nabla f+\alpha \xi)dM=\int_M((\Delta f)^2-|\nabla^2 f|^2-n(n-1)\alpha^2f^2+\alpha^2|\varphi|^2)dM. \nonumber \\
\end{eqnarray}
Using $\Delta f=-n\alpha f,$ we have
\begin{eqnarray}\label{eqthe2p}
\int_M \big(Ric(\nabla f+\alpha \xi,\nabla f+\alpha \xi)-\alpha^2|\varphi|^2\big)dM=\int_M\left(\frac{1}{n}(\Delta f)^2-|\nabla ^2f|^2\right)dM.  \nonumber \\
\end{eqnarray}
Now, using Schwartz’s inequality  in equation \eqref{eqthe2p}, we get
	\begin{equation*}
		\int_{M}Ric(\nabla f+\alpha \xi,\nabla f+\alpha \xi)dM\leq \alpha^2\int_{M}|\varphi|^2dM,
	\end{equation*} which is precisely the inequality (\ref{eqtheorem2}). Moreover, equality holds if and only if
$$|\nabla ^2f|^2=\frac{1}{n}(\Delta f)^2,$$	
which happens if and only if
$$\nabla ^2f=\frac{1}{n}(\Delta f)g.$$	
Since  $\Delta f=-n\alpha f,$ we obtain
$$\nabla^2f=-\alpha f g.$$

Using the Remark~\ref{observation}, we have $\alpha>0.$ Therefore, by the hypothesis on $f$ ($f=0$ in $\partial M$) and for $\nabla^2f=-\alpha f g$ we can apply  Theorem B in \cite{Reilly2} to ensure that $M$ is isometric to a geodesic ball on $\mathbb{S}^n(\alpha)$. Since  $\partial M$ is totally geodesic, we conclude that  $M$ is isometric to a  hemisphere of $\mathbb{S}^n(\alpha)$. 	
\end{proof}

\subsection*{Acknowledgements} The first named author was partially supported by a grant from CNPq-Brazil. The authors would like to thank Professor  Abd\^enago Barros for fruitful conversations about the results.

%%%%%%%%%%% To ease editing, use normal size for the references:

\end{document}